\documentclass{article}
\usepackage{xcolor}
\usepackage{graphicx}
\usepackage{hyperref}
\usepackage{amsfonts,amsmath,amsthm}

\newtheorem{theorem}{Theorem}

\newtheorem{lemma}{Lemma}

\newtheorem{remark}{Remark}
\newtheorem{definition}{Definition}
\newcommand{\expected}{\lambda(u,T)}
\newcommand{\matrixalpha}{\mathbf{\Psi}}
\newcommand{\ci}{\mathcal{J}}
\newcommand{\bu}{\mathbf{u}}
\newcommand{\bv}{\mathbf{v}}
\newcommand{\dr}{\dot{r}}
\newcommand{\ddr}{\ddot{r}}
\newcommand{\dx}{\dot{X}}
\newcommand{\C}{\mathcal{C}}

\newcommand\var{\operatorname{\mathbf{var}}}
\newcommand\cov{\operatorname{\mathbf{cov}}}
\newcommand\cor{\operatorname{\mathbf{cor}}}

\newcommand{\R}{\mathbb{R}}

\newcommand{\X}{\mathcal{X}}

\newcommand\E{{\operatorname{\mathbf{E}}}}

\begin{document}
\title{Asymptotic normality of high level-large time crossings of a Gaussian process}
\author{
F. Dalmao\thanks{
DMEL, 
Universidad de la Rep\'{u}blica, Salto, Uruguay. 
E-mail: fdalmao@unorte.edu.uy.}
\qquad
J. R. Le\'{o}n\thanks{IMERL, Universidad de la Rep\'{u}blica, Montevideo, Uruguay and and Escuela de Matem\'{a}tica. Facultad de Ciencias. 
Universidad Central de Venezuela, Caracas, Venezuela. 
E-mail: jose.leon@ciens.ucv.ve }
\qquad
E. Mordecki\thanks{CMAT, 
Universidad de la Rep\'{u}blica, Montevideo, Uruguay. 
E-mail: mordecki@cmat.edu.uy.}
\qquad 
S. Mourareau\thanks{LAMA, 
Universit\'e Paris-Est Marne la Vall\'ee, France. 
E-mail: stephane.mourareau@u-pem.fr}
}

\maketitle

\begin{abstract}
We prove the asymptotic normality of the standardized number of crossings of a 
centered stationary mixing Gaussian process 
when both the level and the time horizon go to infinity 
in such a way that the expected number 
of crossings also goes to infinity. \\
\end{abstract}



\section{Introduction}
% # 1 - on the number of crossings
The number of crossings of a stochastic process through a level over a time interval 
gives important information about the geometry of the trajectories and, 
for large values of the level, about the behavior of the tail of the maximum of the process. 
Since the exact distribution of this functional is not known but in very particular cases,
the study of its asymptotic behavior, under different asymptotic schemes, 
has become a classical subject of research.

% # 2 - %% volkonskii rozanov
In the case when the level $u$ and the time horizon $T$ both go to infinity in such a way that
the expectation of the number of crossings through $u$ on $[0,T]$ remains fixed, the classical results
by Volkonski\u{\i} and Rozanov state that the conveniently normalized number of up-crossings
asymptotically behaves like a standardized Poisson process (\cite{vr1},\cite{vr2}) where the intensity 
is the constant expectation of the number of up-crossings.
Since the standardized Poisson distribution approximates the normal distribution when this intensity tends to infinity, 
it is natural to ask for the asymptotic normality when %the mean number of zeros 
its intensity 
tends to infinity.

% # 3 - %% moments rice
A classical way to prove asymptotic normality is based on the computation
of different moments of the underlying random variables.
In the case of the number of crossings of a smooth stochastic process, 
this task can be carried out with the help of Rice formulas (\cite{rice}). 
However, the explicit computation and the analysis of the asymptotic behavior of moments of higher order than the second 
is in general a difficult task, in particular taking into account that the level is not fixed.
A general picture of the field can be found in the books by Cram\'er and Leadbetter \cite{cl} and Aza\"is and Wschebor \cite{aw}.
% # 4 - %% chaos expansions
An alternative approach to prove asymptotic normality consists in the study of the chaotic expansion of the crossings in the Wiener space (see for instance \cite{slud:1991}),
using the corresponding limit theorem results as the one exposed in the books
by Peccati and Taqqu \cite{pt} or Nourdin and Peccati \cite{np}. 
Wiener chaos techniques have the advantage of avoiding higher moments than the second 
and, sometimes, of giving rates of convergence.
% # 5 - %% first results in chaos
The first results in this direction were obtained by Malevi\v{c} \cite{mal}, Cuzick \cite{c}, Slud \cite{slud} and 
Kratz and Le\'on  \cite{kl} where 
-within other results- 
the normal asymptotic behavior of the standardized number of crossing of a smooth stationary Gaussian process is obtained for a fixed level $u$ as the time horizon goes to infinity. 
In order to obtain this result,  the chaotic expansion of the 
number of 
crossings and the approximation of the process by $m$-dependent processes are used. 
% # 6 - %% Mourareau
In his recent Phd thesis \cite{m},
Mourareau analyzes the chaos expansion of the crossings in the case where the level, 
the time horizon and the mean number of crossings go to infinity, 
obtaining the asymptotic normality of the normalized number of crossings for an $m$-dependent Gaussian process. 
However, the usual scheme of translating this approximation to 
more general Gaussian processes 
is not carried out. 
Mourareau's work points out the sophisticated nature of this situation. 
While most of the asymptotic distributions found via the Wiener chaos techniques 
rely on the fact that 
some -possibly every- 
component of the chaotic expansion 
of the functional have variance of the same order than the global variance, 
in the present case the variance of each chaotic component of the number of crossings tends to zero. 

% # 6 - our paper
In the present paper, we study the normal asymptotic behavior of the normalized number of crossings 
of a class of stationary Gaussian processes 
when both the level and the observation time go to infinity, 
in such a way that the expectation of the number of crossings also goes to infinity (see Theorem \ref{teo}). 
The basic idea to obtain our result is to use the Bernstein block method \cite{bernstein} for dependent random variables
in the Central Limit Theorem, with the formulation presented in \cite{bdlr}.
In particular, this approach requires the analysis of the asymptotic behavior of the second and third moments of
the number of crossings over an increasing time interval, when the level of the crossings also goes 
to infinity in a regulated way. This task is accomplished with the help of the corresponding Rice formulas.
The computations of the third moment for zero level (i.e. roots) in the context of stationary random polynomials can be found in \cite{gw}, see also section 5.2 in \cite{aw}.
To our knowledge, the use of the Rice formula for the third factorial moment at an arbitrary level 
has not been used previously. In the way of our proof, we obtain the equivalence of the asymptotic behavior
of the expectation and the variance of the crossings under a very general asymptotic scheme (see Theorem \ref{key}).

% # 7 - % Notations
We use the usual notations $f(t)\sim g(t)$ to indicate that $\lim\frac{f(t)}{g(t)}=1$; 
$f(t)=o(g(t))$ for $\lim\frac{f(t)}{g(t)}=0$. 

% # 8 - % contents
The rest of the paper is organized as follows: 
Section \ref{main} presents the problem, the main results 
and some motivating partial results. 
Section \ref{preliminaries} introduces some preliminary results. 
Section \ref{theorem1} contains the proof of Theorem 1, 
Section \ref{theorem2} the proof of Theorem 2 and
Section \ref{theorem3} the proof of Theorem 3.
%The computation of the third factorial moment through Rice formula is deferred to appendix \ref{app1}.
% In a second appendix \ref{app2} we present a partial related result.

%
\section{Main results}\label{main}
% the process
Assume that $\X=\{X(t)\colon t\in\R\}$ is a mean zero variance one stationary Gaussian process
with smooth paths. Denote the covariance function of $\X$ by 
$$r(\tau)=\E(X(t)X(t+\tau)).$$ 
Without loss of generality we assume that 
$r(0)=1$, and that
$\var\dx(0)=-{\ddr}(0)=1.$
% crossings
Define the number of crossings through level $u$ by the process $\X$ over the time interval $I\subset\R$ by
\begin{equation*}
N(I,u)=\#\{t\in I\colon X(t)=u\}, 
\end{equation*}
and denote $N(T,u)=N([0,T],u)$. 
For simplicity of notation we always assume $u>0$.
Set 
\begin{align*}
C(u)&=\E\, N(1,u)={1\over \pi}\exp(-u^2/2),\\
\lambda(T,u)&=\E\, N(T,u)=TC(u)={T\over \pi}\exp(-u^2/2).
\end{align*}

% main results
We present now our main results. 
The first theorem states that 
when the level $u$ tends to infinity, 
uniformly in $T$ bounded away from zero, 
the mean and the variance of $N(T,u)$ 
are of the same order. 
Wiener chaos techniques are used in Proposition 7.4.2 of \cite{m} 
to prove a similar result 
under a more restrictive time-level asymptotic scheme.
Here we use an approach based on Rice formula.

\newpage
\begin{theorem}\label{key}
Assume that $\X$ satisfies Geman's condition:
\begin{itemize}
\item[{\rm (G)}] If ${r}(\tau)=1-\tau^2/2+\theta(\tau)$, then, for some $\delta>0$, 
the integral
$$
\int_0^{\delta}{\theta^{\,\prime}(\tau)\over\tau^2}d\tau<\infty.
$$ 
\end{itemize}
Furthermore, assume that $r(\tau)\to 0$, $\dr(\tau)\to 0$ as $\tau\to\infty$, 
and that the integral
\begin{equation}\label{eq:integrable}
\int_0^{\infty}\big(|r(\tau)|+|\dr(\tau)|+|\ddr(\tau)|\big)d\tau<\infty.
\end{equation}
Hence, for any fixed $t_0>0$,  
as $u$ tend to infinity, 
we have
\begin{equation}\label{eq:key}
{\var N(T,u)\over \E\, N(T,u)}\to 2,
\end{equation}
uniformly in $T\in[t_0,\infty)$. 
\end{theorem}

The second result states the 
asymptotic normality of the standardized number of crossings $N(T,u)$.  
Extra conditions on the process $\X$ are imposed 
and now $u$ depends on $T$, 
we write $u_T$ to emphasize this dependence. 
We need the following definition. 
\begin{definition}\label{remark:a1}
For $t>0$ and $-\infty\leq a\leq b\leq \infty$, let $\mathcal F_a^b$ be the $\sigma$-algebra generated by the random variables $\{X_s\colon a\leq s\leq b\}$. 
The $\alpha$-mixing coefficient is defined as
$$\alpha(t)=\sup\{|\mathbf{P}(U)\mathbf{P}(V)-\mathbf{P}(U\cap V)|:\, U\in\mathcal F_{-\infty}^0,\, V\in\mathcal F_{t}^{\infty}\},$$
and the $\rho$-mixing coefficient is 
$$\rho(t)=\sup\{|\cov(X,Y)|:\, X\in\mathbb L^2(\Omega,\mathcal F_{-\infty}^0,\mathbf{P}),\,Y\in\mathbb L^2(\Omega,\mathcal F_{t}^{\infty},\mathbf{P})\}.
$$
We say that the process is $\alpha$-mixing (resp. $\rho$-mixing) if $\alpha(t)\to_{t\to\infty}0$ (resp. $\rho(t)\to_{t\to\infty}0$).
\end{definition}
\begin{theorem}\label{teo}
Assume the following conditions on the process $\X$: 
For any $t_1<t_2<t_3$, the distribution of the vector $(X(t_1),X(t_2),X(t_3))$ is non-degenerated. 
The process $\X$ is $\alpha$-mixing with a polynomial rate $\mu>4$, i.e. 
\begin{equation}\label{eq:mixing}
\alpha(T)\sim T^{-\mu},\quad (T\to\infty).
\end{equation}
The covariance function verifies:
\begin{equation}\label{eq:expansion}
r(\tau){ =} 1-\frac12\tau^2+d\tau^4+e\tau^6+o(\tau^6),\quad (\tau\to0).
\end{equation}
Besides, assume that there exists $0<\gamma<1$ such that $\mu\gamma>4$ and 
$$
\expected\sim T^\gamma,\quad (T\to\infty).
$$
Then, the standardized number of crossings converges in distribution towards 
the standard normal distribution as  $T\to\infty$, that is:
\begin{equation}\label{eq:main}
{N(T,u_T)-\expected\over\sqrt{2\expected}}\Rightarrow \mathcal N(0,1),\quad(T\to\infty).
\end{equation}
\end{theorem}
\begin{remark}
A sufficient condition for the non-degeneracy of the finite-distributions of the stationary Gaussian process $\X$ 
is that the support of the spectral measure of $r$ has an accumulation point, see page 82 of \cite{aw}.
\end{remark}

%%%%
%%%%%%%%%%%%%%%%%%%%%

Let us finish this section discussing a direct approach based on 
the celebrated Volkonski\u{\i}-Rozanov Theorem.
This approach yields some partial results and the motivation for Theorem \ref{teo}, 
see \cite{vr1,vr2}  or Th. 10.1 in \cite{aw}. 
Roughly speaking, Volkonski\u{\i}-Rozanov Theorem states that when 
the number of up-crossings 
$$
U(T,u)=\#\{t\in [0,T]:X(t)=u,X'(t)>0\},
$$
satisfy $\E\, U(T,u)\to\ell$,
once normalized, they converge towards a Poisson distribution with parameter $\ell$. 
As said above, the motivation for Theorem \ref{teo} is provided by the fact that 
the normalized Poisson distribution with parameter $\ell$ converges to the normal distribution as $\ell\to\infty$. 
Theorem \ref{teo:vr} below, using this approach, 
states the asymptotic normality of $N(T,u)$ but only for some sequences $(T_n,u_n)$ 
in a non-constructive way. 
Theorem \ref{teo} gives a more satisfactory result under more restrictive conditions on the process $\X$. 
In order to formulate Theorem \ref{teo:vr}, whose hypothesis are the same as those of Volkonski\u{\i}-Rozanov Theorem (Th. 10.1 in \cite{aw}),
we need to introduce the following condition:
\begin{itemize}
\item[(B)] Berman's condition: ${r}(\tau)\log(\tau)\to 0$ as $\tau\to\infty$. 
\end{itemize}
\begin{theorem}[Existence of a good sequence]
Assume that $\X$ satisfies conditions $-\ddr(0)<\infty$, \rm{(B)} and \rm{(G)}.
Then, there exists a sequence $(T_n,u_n)\to(\infty,\infty)$ such that $\lambda_n:=\lambda(T_n,u_n)\to\infty$ 
and
\begin{equation*}
\frac{N(T_n,u_n)-\lambda_n}{\sqrt{2\lambda_n}}\Rightarrow \mathcal{N}(0,1),\text{ as $n\to\infty$.}
\end{equation*}
\label{teo:vr}
\end{theorem}
\begin{remark}\label{remark:up}
As the process has continuous trajectories, we have the relation 
$|N(T,u)-2U(T,u)|\leq 1$,
and all asymptotic results for large number of crossings can be expressed in terms of crossings or up-crossings.
Observe that this is not the case in Volkonskii-Rozanov scheme.
\end{remark}
\section{Preliminaries}\label{preliminaries}
\subsection{The covariance function and its derivatives}
We present some basic results that are used thorough the paper.
\begin{lemma} \label{lemma:covariance}
Assume that the process $\X$ has a twice differentiable covariance.
\par\noindent {\rm (a)} Under the conditions $r(0)=-\ddr(0)=1$, we have $|r(t)|\leq 1$, $|\dr(t)|\leq 1$, $|\ddr(t)|\leq 1$.
\par\noindent {\rm (b)} Under \eqref{eq:integrable} the functions $r$, $\dot{r}$ and $\ddot{r}$ 
are in $L^p(\R)$ for all $p\geq 1$.
\par\noindent {\rm (c)}
The $\alpha$-mixing condition 
\eqref{eq:mixing} implies that  the functions $r$, $\dot{r}$ and $\ddot{r}$ 
are in $L^p(\R)$ for all $p\geq 1$. In particular, condition \eqref{eq:integrable} holds.
\end{lemma}
\begin{proof}
{\it Proof of {\rm (a)}.}
First observe that, according to Proposition 1.13 in \cite{aw}, the process $\{\dot X(t)\}$ exists, as the derivative in quadratic mean of $\X$.
We have
$$
r(t)=\E(X(0)X(t)),\quad \dr(t)=\E(X(0)\dot X(t)),\quad  \ddr(t)=-\E(\dot X(0)\dot X(t)).
$$
As $\E(X(t)^2)=\E(\dot X(t)^2)=1$, the statements of (a) follow from the application of the Cauchy-Schwarz inequality.
\par\noindent
{\it Proof of {\rm (b)}.} In view of (a) it is direct.
\par\noindent
{\it Proof of {\rm (c)}.}
By using that $X(0)$ and $\dot X(0)$ belong to $\mathcal F_{-\infty}^0$, and that $X(t)$ and $\dot X(t)$ belong to $\mathcal F_t^{\infty}$, 
we get
$$
|r(t)|\le\rho(t),\quad |\dr(t)|\le \rho(t),\quad |\ddr(t)|\le\rho(t).
$$ 
As by Kolmogorov-Rozanov inequality (see pag. 57 in \cite{mixing}), 
$\rho(t)\le \mathbf C\alpha(t)$ and that by our hypothesis $\int_0^\infty\alpha(t)dt<\infty,$ we get that the three functions $r,\dr,\ddr$ are integrable.  
As by (a) they are bounded, the statement of (c) follows.
\end{proof}
\begin{remark}\label{remark:a2}
The parameters in \eqref{eq:expansion} are not completely free, in particular 
\begin{equation*}%\label{eq:jensen}
r^{(iv)}(0)-\ddr(0)^2=24d-1\geq 0,
\end{equation*}
In fact, 
$\ddr(t)=\E(X(0)\ddot X(t))$, 
so taking $t=0$, and applying Cauchy-Schwarz inequality, we obtain 
$|\ddr(0)|=|\E(X(0)\ddot X(0))|\le \sqrt{r(0)r^{(iv)}(0)}$.
Since $r(0)=\ddr(0)=1$  
we get the desired inequality.
\end{remark}

\subsection{Rice formula}
In the sequel we need to deal with the second and the third moments of the number of crossings $N=N(I,u)$. 
The main tool is the celebrated Rice formula which we now present in its general form, see \cite{aw} for the details. 

Let $n\geq 1$, $N^{[n]}(T,u)=\prod^{n-1}_{k=0}(N(T,u)-k)$, then 
we have
\begin{equation}\label{eq:rice}
 \E\big(N^{[n]}(T,u)\big)
 =\int_{I^{n}}\E\left[\left|\prod^n_{k=1}\dot{X}(s_k)\right|\mid X=\bu\right]p_X(\bu)ds,
\end{equation}
being $s=[s_1,\dots,s_n]$, $X=[X(s_1),\dots,X(s_n)]^t$ 
and $\bu=[u,\dots,u]^t$. 
To deal with the conditional expectation in this formula we use 
the following result, that has a direct proof.
%On the regression in $\R^n$
\begin{lemma}[Gaussian regression]\label{lemma:1}
Consider the times $s_1\leq\cdots\leq s_n\in I$. 
Denote $\dx=[\dx(s_1),\dots,\dx(s_n)]^t$, 
$$
%\aligned
\Sigma=\var X=\E(XX^t),\quad
\Sigma_{10}=\cov(\dx,X)=\E(\dx X^t),\quad
\Sigma_{11}=\E(\dx\dx^t)=\var\dx,
$$
and
$$
\matrixalpha
=
\left[
\begin{matrix}
\Psi_1& \dots &\Psi_n
\end{matrix}
\right]^t
:=\Sigma_{10}\Sigma^{-1}.
$$
Then:
{\rm (a)} The Gaussian vector $Y=[Y_1,\dots,Y_n]^t$ defined by
\begin{equation}\label{eq:pitagoras}
Y=\dx-\matrixalpha X
\end{equation}
is independent from $X$, centered, and has a variance matrix given by
\begin{equation*}
\Sigma_Y=\var Y=\Sigma_{11}+\matrixalpha\Sigma_{10}=\Sigma_{11}-\matrixalpha\Sigma^{t}_{10},
\end{equation*}
since $\Sigma^{t}_{10}=-\Sigma_{10}$. 
Since $\Psi \Sigma^{t}_{10}=\Sigma_{10}\Sigma^{-1}\Sigma^{t}_{10}$ 
and $\Sigma^{-1}$ is positive definite, it follows from \eqref{eq:pitagoras} that 
$\var Y_i\leq \var \dx_i=1$ for $i=1,2,\dots,n$.

\noindent {\rm(b)} 
The random vector
$$
(Y_1+\Psi_1\bu,\dots,Y_n+\Psi_n\bu),
$$
has the same distribution as $\dx$ conditional on the set $\{X=\bu\}$.
\end{lemma}
%%%%%%%%%%%%%%%%%%%%%%%%%%%%%%%
\section{Proof of Theorem \ref{key}}\label{theorem1}
In this section we prove that the asymptotic expectation and variance of $N(T,u)$ are of the same order. 
We begin by specializing Lemma \ref{lemma:1} to the case $n=2$, see pg. 76 in \cite{aw}. 
For $\tau=s_2-s_1$ we have
\begin{equation*}
\Sigma=
\left[\begin{array}{cc}1&r(\tau)\\r(\tau)&1\end{array}\right].
\end{equation*}
Thus --ommiting the $\tau$ in the notation--, we have
$$
\Sigma_{10}=
\dr
\left[
\begin{array}{cc}
0&-1\\ 
1&0
\end{array}
\right],\qquad
\Sigma_{11}=
\left[
\begin{array}{cc}
1&-\ddot{r}\\-\ddot{r}&1
\end{array}
\right].
$$
Hence,
$$
\matrixalpha=
\left[\begin{matrix}
\Psi_1\\ \Psi_2
\end{matrix}\right]
=\frac\dr{1-r^2}
\left[\begin{matrix}
r & -1\\ 
1 & -r\\
\end{matrix}\right],\quad
\var\matrixalpha X=\frac{\dr^2}{1-r^2}\Sigma.
$$
$$
\var Y=
\left[
\begin{array}{cc}
1&-\ddot{r}\\-\ddot{r}&1
\end{array}
\right]-\frac{\dr^2}{1-r^2}
\left[\begin{array}{cc}1&r\\r&1\end{array}\right].
$$
Therefore, we find the following expressions %asymptotic expansions as $\tau\to 0$:
\begin{align}\label{eq:rho}
\var Y_1&=\var Y_2=1-\frac{\dr^2}{1-r^2},\quad%\notag\\ \quad %\sim \left(\frac{24d-1}4\right)\tau^2\ ,\notag\\
\cov(Y_1,Y_2)=-\ddr-\frac{r\dr^2}{1-r^2},\notag\\ %\sim -1-24d\tau+\frac14(24d+1)\tau^2,\notag\\
\rho:&={\cov(Y_1,Y_2)\over \var Y_1}=-\frac{(1-r^2{})\ddot{r}{}+r{}\dot{r}^2{}}{1-r^2{}-\dot{r}^2{}}.
\end{align}
We also have
$$
\Psi\bu=\frac{\dr}{1+r}
\left[
\begin{matrix}
-u\\ 
u\\
\end{matrix}
\right]
\sim 
\frac{-\tau}2\left[
\begin{matrix}
-u\\ 
u\\
\end{matrix}
\right],\ (\tau\to 0).
$$
Then we denote 
\begin{equation}\label{eq:psiu}
\psi u=\Psi_1\bu=-\Psi_2 \bu=\frac{\dr}{1+r}u.
\end{equation}

%%%%%%%%%%%%%%%%%%%%%%%%%%%%%%%
%\subsection{On the rates of  the variance and expectation}
Now we turn to the proof of Theorem \ref{key}. 
\begin{proof}[Proof of Theorem \ref{key}]
With $N=N(T,u)$, we have
\begin{equation}\label{eq:var}
\frac{\var N}{\E N}=\frac{\E(N(N-1))-\E^2 N}{\E N}+1. 
\end{equation}
The condition $r(\tau)\to0$ ($\tau\to\infty$) implies that $|r(\tau)|\neq1$ for $\tau\neq0$. 
Hence, we can apply Rice formula. 
By Rice formula and the stationarity of $\X$ we have
\begin{equation*}
\E(N(N-1))=2\int^T_0(T-\tau)\E(|\dot{X}(0)\dot{X}(\tau)|\mid X(0)=X(\tau)=u)p_{X(0),X(\tau)}(u,u)d\tau. 
\end{equation*}
Here,
\begin{equation*}
p_{X(0),X(\tau)}(u,u)=\frac{1}{2\pi}\frac{1}{\sqrt{1-r^2(\tau)}}e^{-\frac{u^2}{1+r(\tau)}}.
\end{equation*}
% Since $X(t)$ is standard Gaussian for each $t$, we can write
Also by Rice formula $\E N=TC(u)=T\E(|\dot{X}(0)|)p_{X(0)}(u)$, 
since $T^2=2\int^T_0(T-\tau)d\tau$, we have 
\begin{equation*}
\E^2 N
=2\int^T_0(T-\tau)\E^2(|\dot{X}(0)|)p^2_{X(0)}(u)d\tau.
\end{equation*}
Hence, using the Gaussian regression in Lemma \ref{lemma:1}, 
the first term in the r.h.s. of \eqref{eq:var} can be written as $\int^T_0 \ci(T,\tau,u)d\tau$, 
being
$$
\ci(T,\tau,u)={T-\tau\over T}
\left(\E(|Y_1+\psi u||Y_2-\psi u|)\frac{e^{-\frac{u^2}{2}\frac{1-r(\tau)}{1+r(\tau)}}}{\sqrt{1-r^2(\tau)}}
-\E^2(|\dot{X}(0)|)e^{-{u^2\over 2}}\right).
$$
We have to prove that
$$
\lim_{u\to\infty}\int_0^T\ci(T,\tau,u)d\tau ={1},
$$
uniformly on $T\geq t_0$ for any $t_0>0$.
We divide the proof in three steps, corresponding to small, medium and large values of $\tau$ in the integral.

%%%%%%%%%%%%%%%%%%%%%%%%%%%%%%
%%%     FIRST STEP         %%%
%%%%%%%%%%%%%%%%%%%%%%%%%%%%%%
\par\vspace{2mm}\noindent
\textit{First step.}  For an arbitrary small $\delta>0$ to be chosen, smaller than $t_0$, 
uniformly in $T$, we have
$$
\lim_{u\to\infty}\int_0^\delta\ci(T,\tau,u)d\tau ={1}.
$$
It is no hard to verify the following limits for any $\delta>0$:
\begin{align*}
\lim_{u\to\infty}&\int_0^\delta{T-\tau\over {} T}
\E^2(|\dot{X}(0)|)e^{-{u^2\over 2}}d\tau=0,\\
\lim_{u\to\infty}&\int_0^\delta{T-\tau\over {} T}
\E|Y_1Y_2|
\frac{e^{-\frac{u^2}{2}\frac{1-r(\tau)}{1+r(\tau)}}}{\sqrt{1-r^2(\tau)}}
d\tau=0,\\
\lim_{u\to\infty}&\int_0^\delta{T-\tau\over {} T}
\E|(Y_1-Y_2)\psi u|
\frac{e^{-\frac{u^2}{2}\frac{1-r(\tau)}{1+r(\tau)}}}{\sqrt{1-r^2(\tau)}}
d\tau=0.
\end{align*}
Note that Geman's condition guarantees the finiteness of the second integral 
since $\frac{\theta'(\tau)}{\tau^2}\sim\frac{\var{Y_i}}{\sqrt{1-r^2(\tau)}}$ as $\tau\to0$, see page 99 of \cite{aw}. 
From these results, it follows that
$$
\lim_{u\to\infty}\int_0^\delta\ci(T,\tau,u)d\tau=
\lim_{u\to\infty}\int_0^\delta{T-\tau\over {} T}
(\psi u)^2
\frac{e^{-\frac{u^2}{2}\frac{1-r(\tau)}{1+r(\tau)}}}{\sqrt{1-r^2(\tau)}}
d\tau.
$$
Now, using \eqref{eq:psiu}, we have to prove that
$$
\lim_{u\to\infty}\int_0^\delta{T-\tau\over {} T}
\left({\dr(\tau)u\over1+r(\tau)}\right)^2
\frac{e^{-\frac{u^2}{2}\frac{1-r(\tau)}{1+r(\tau)}}}{\sqrt{1-r^2(\tau)}}
d\tau={1}.
$$
We now observe that, as 
$$
\lim_{\tau\to 0}{T-\tau\over  T}
{-\dr(\tau)\over\sqrt{1-r^2(\tau)}}=1,
$$
and $\delta$ is arbitrarily small, we have to prove that
$$
\lim_{u\to\infty}\int_0^\delta
{-\dr(\tau)u^2\over(1+r(\tau))^2}
e^{-\frac{u^2}{2}\frac{1-r(\tau)}{1+r(\tau)}}
d\tau={1}.
$$
We change variables according to 
$$
v={u^2\over 2}\left(\frac{1-r(\tau)}{1+r(\tau)}\right),
$$
that is monotonous for $\delta$ conveniently small. 
So
\begin{equation*}
\int_0^\delta
{-\dr(\tau)u^2\over(1+r(\tau))^2}
e^{-\frac{u^2}{2}\frac{1-r(\tau)}{1+r(\tau)}}
d\tau
=\int_0^{{u^2\over 2}\left(\frac{1-r(\delta)}{1+r(\delta)}\right)}
e^{-v}
dv
=1-\exp\left(-
{u^2\over 2}\left(\frac{1-r(\delta)}{1+r(\delta)}\right)
\right)\to 1\ (u\to\infty),
\end{equation*}
concluding the first step.

%%%%%%%%%%%%%%%%%%%%%%%%%%%%%%
%%%     SECOND STEP       %%%
%%%%%%%%%%%%%%%%%%%%%%%%%%%%%%
\par\vspace{2mm}\noindent
\textit{Second step.}  For all $T_0>0$ 
\begin{equation}\label{eq:second}
\lim_{u\to\infty}\int_\delta^{T_0}\ci(T,\tau,u)d\tau =0.
\end{equation}
We now work for $\tau\geq\delta>0$. As both $r(\tau)$ and $\dr(\tau)$ converge to zero
as $\tau\to\infty$, are continuous functions, and $|r(\tau)|\neq 1$ for $\tau>0$, 
we obtain the existence a constant,
say $r_0<1$, $|r(\tau)|\leq r_0$. As $|\dr(\tau)|\leq 1$, we have 
\begin{equation*}
\E(|\dot{X}(0)\dot{X}(\tau)|\mid X(0)=X(\tau)=u)=
\E(|Y_1+\psi u||Y_2-\psi u|)
\leq 1+{2\over 1-r_0} u+\left({1\over 1-r_0}\right)^2u^2\leq A+B u^2.
\end{equation*}
It follows that, for all fixed $\tau>\delta$ 
and uniformly in $T$, we have
\begin{equation*}
\lim_{u\to\infty}\ci(T,\tau,u)=0.
\end{equation*}
Furthermore,
$$
\frac{e^{-\frac{u^2}{2}\frac{1-r(\tau)}{1+r(\tau)}}}{\sqrt{1-r^2(\tau)}}\leq 
\frac{e^{-\frac{u^2}{2}\frac{1-r_0}{1+r_0}}}{\sqrt{1-r^2_0}},
$$
so the integrand in \eqref{eq:second} is uniformly bounded in $u$, concluding the second step by dominated convergence. 

%%%%%%%%%%%%%%%%%%%%%%%%%%%%%%
%%%     Third STEP        %%%
%%%%%%%%%%%%%%%%%%%%%%%%%%%%%%
\par\vspace{2mm}\noindent
\textit{Third step.}
There exists $T_0>0$ such that
$$
\lim_{u\to\infty}\int_{T_0}^\infty\ci(T,\tau,u)d\tau =0.
$$
Coming back to the difference, 
\begin{multline*}
\int_{T_0}^\infty\ci(T,\tau,u)d\tau
=\int_{T_0}^\infty
\frac{T-\tau}{T}
\frac{e^{-\frac{u^2}{2}\frac{1-r(\tau)}{1+r(\tau)}}}{\sqrt{1-r^2(\tau)}}
\Big(\E(|Y_1+\psi u||Y_2-\psi u|)-\E^2(|\dot{X}(0)|)\Big)d\tau\\
+\frac2\pi\int_{T_0}^\infty\frac{T-\tau}{T}
\Big(
\frac{e^{-\frac{u^2}{2}\frac{1-r(\tau)}{1+r(\tau)}}}{\sqrt{1-r^2(\tau)}}
-e^{-\frac{u^2}2}\Big)d\tau
=:I_1+I_2,
\end{multline*}
where $I_1,I_2$ denote respectively the first and second addends.
%%%%
Now, we apply the triangle inequality. 
For the second integral we get
\begin{equation*}
I_2\leq \frac2\pi\int_{T_0}^\infty
 \Bigg|\frac{e^{-\frac{u^2}{2}\frac{1-r(\tau)}{1+r(\tau)}}}{\sqrt{1-r^2(\tau)}}-e^{-\frac{u^2}{2}}\Bigg|d\tau.
\end{equation*}
We have
\begin{equation}\label{eq:first}
\frac{e^{-\frac{u^2}{2}\frac{1-r(\tau)}{1+r(\tau)}}}{\sqrt{1-r^2(\tau)}}-e^{-\frac{u^2}{2}}
=\frac{e^{-\frac{u^2}{2}\frac{1-r(\tau)}{1+r(\tau)}}-e^{-\frac{u^2}{2}}}{\sqrt{1-r^2(\tau)}}
+\left[\frac{1}{\sqrt{1-r^2(\tau)}}-1\right]e^{-\frac{u^2}{2}}.
\end{equation}
The second term 
is equivalent, as $\tau\to\infty$, to $\frac12 r^2(\tau)e^{-u^2/2}$, because $r(\tau)\to 0$ as $\tau\to\infty$. 
In conclusion its integral vanishes because $\int_0^{\infty}r^2(\tau)d\tau<\infty$ (see (b) in Lemma \ref{lemma:covariance}).
For the first term in the r.h.s in \eqref{eq:first}, we
denote $\beta=\frac{1-r(\tau)}{1+r(\tau)}$.
As $\beta\to 1$, we take $T_0$ large enough such that $|\beta - 1|\leq 1/2$, and, simultaneously,
as $r(\tau)\to 0$, such that $\sqrt{1-r^2(\tau)}\geq 1/2$, and also $|r(\tau)|\leq 1/2$, these three inequalities for all $\tau\geq T_0$.
We apply Lagrange's Formula to the difference inside the integrand, 
$$
e^{-\frac{u^2}{2}\beta}-e^{-\frac{u^2}{2}}=e^{-\frac{u^2}{2}\theta}\left(-\frac{u^2}{2}\right)
\left(\frac{1-r(\tau)}{1+r(\tau)}-1\right)=e^{-\frac{u^2}{2}\theta}\frac{u^2r(\tau)}{1+r(\tau)}.
$$
The value $\theta$ is such that $|\theta-1|\leq|\beta-1|\leq 1/2$, so $\theta\geq 1/2$. Then
$$
e^{-\frac{u^2}{2}\theta}\leq e^{-\frac{u^2}{4}}.
$$
In conclusion, we have the following bound for the integrand:
$$
\frac1{\sqrt{1-r^2(\tau)}}\Big|e^{-\frac{u^2}{2}\frac{1-r(\tau)}{1+r(\tau)}}-e^{-\frac{u^2}{2}}\Big|
\leq 2e^{-\frac{u^2}{4}}u^2|r(\tau)|,
$$
and, as $r(\tau)$ is integrable, we conclude that 
$$
\int^\infty_{T_0} \Bigg|\frac{e^{-\frac{u^2}{2}\frac{1-r(\tau)}{1+r(\tau)}}-e^{-\frac{u^2}{2}}}{\sqrt{1-r^2(\tau)}}\Bigg|d\tau
\leq 
2e^{-\frac{u^2}{4}}u^2\int^\infty_{T_0}|r(\tau)|d\tau\to 0\quad({ u}\to\infty).
$$
Let us look at the first integral $I_1$. With similar arguments we have the following bound:
$$ 
I_1\leq 2e^{-\frac{u^2}{4}}\int_{T_0}^{\infty}
\Big|\E(|Y_1+\psi u||Y_2-\psi u|)-\E^2(|\dot{X}(0)|)\Big|d\tau
$$
Denote $\sigma^2=\var Y_1=\var Y_2$, and $Z_i=Y_i/\sigma$ $i=1,2$.
We have
\begin{align*}
\big|
\E(|Y_1+\psi u| 	&|Y_2-\psi u|)-\E^2(|\dot{X}(0)|)\big|
%\\&
\leq
	\left|\E(|Y_1+\psi u||Y_2-\psi u|)-\E(|Y_1+\psi u||Y_2|)\right|
	\\&\quad
+	\left|\E(|Y_1+\psi u||Y_2|)-\E(|Y_1||Y_2|)\right|
+	\left|\E(|Y_1||Y_2|)-\sigma^2\E(|Z_1|)\E(|Z_2|)\right|
+	\left|1-\sigma^2\right|\E^2(|\dot{X}(0)|)\\
&\leq\E(|Y_1+\psi u||\psi u|)
+\E(|Y_2||\psi u|)
+\sigma^2\left|\E(|Z_1Z_2|)-\E(|Z_1|)\E(|Z_2|)\right|
+	\left|1-\sigma^2\right|\E^2(|\dot{X}(0)|).
\end{align*}
Let us analyze term by term. As for $i=1,2$, we have
$$
|\psi u|={|\dr|\over 1+r}u\leq 2|\dr|u. 
$$
Then, for the first two terms
$$
\aligned
\E(|Y_1+\psi u||\psi u|)&\leq \E(|Y_1|)|\psi u|+|\psi u\psi u|
\leq 6|\dr|u^2,\\
\E(|Y_2||\psi u|)&\leq 2|\dr|u.
\endaligned
$$
For the third term, we have $\sigma^2\leq 1$ and 
denoting $\rho=\cor(Y_1,Y_2)$ and $Z$ a standard gaussian random variable, 
independent from $Z_1$
\begin{multline*}
\left|\E(|Z_1Z_2|)-\E(|Z_1|)\E(|Z|)\right|=
\left|\E\left(|Z_1||\rho Z_1+\sqrt{1-\rho^2}Z|\right)-\E(|Z_1|)\E(|Z|)\right|\\
=\left|\E\left(|Z_1|\left|\rho Z_1+\sqrt{1-\rho^2}Z\right|-|Z_1||Z|\right)\right|
\leq
\left|\E\left(|Z_1|\left|\rho Z_1+\left(\sqrt{1-\rho^2}-1\right)Z\right|\right)\right|\\
\leq
|\rho|\E(Z_1^2)+
\left(\sqrt{1-\rho^2}-1\right)\E^2(|Z|).
\end{multline*}
Besides, using equation \eqref{eq:rho} 
\begin{align*}
|\rho|&\leq 2 \left(|\ddr| +\dr^2\right),\\
\left|\sqrt{1-\rho^2}-1\right|&=
\frac{\rho^2}{\sqrt{1-\rho^2}+1}
\leq |\rho|.
\end{align*}
Gathering all the terms, and observing that the integrals of $\dr$ and $\ddr$ are absolutely convergent, 
enlarging $T_0$ if necessary,  
we obtain  the bound
$$
I_1\leq K u^2e^{-\frac{u^2}{4}}\to 0,\quad (u\to\infty),
$$
where $K$ is a convenient constant. This concludes the proof of Theorem \ref{key}.
\end{proof}

%%%%%%%%%%%%%%%%%%%%%%%%%%%%%%%%
%%%%%%%%%%%%%%%%%%%%%%%%%%%%%%%%
%%%%%%%%%%%%%%%%%%%%%%%%%%%%%%%%
\section{Proof of Theorem \ref{teo}}\label{theorem2}
We begin by the construction of the corresponding blocks in Bernstein's scheme.
Consider $b,c$ such that $0<c<b<\gamma/4$ and 
$\mu b>1$. Define
$$
n=n_T=T^{1-b},\quad
p_T=T^b,\quad
q_T=T^c.
$$
Then,  as $T\to\infty$,  
$n_T,p_T,q_T\to\infty$,
$$
n_T(p_T+q_T)=T+{o}(T),
$$
and
\begin{equation}\label{eq:rates}
{n_Tp_T\over T}\to 1,\qquad {n_Tq_T\over T}\to 0.
\end{equation}
For $k=1,\dots,n_T$ denote
$$
I_k=[(k-1)(p_T+q_T),kp_T+(k-1)q_T],\quad
J_k=[kp_T+(k-1)q_T,k(p_T+q_T)].
$$
$$
X^{n_T}_k={N(I_k,u)-p_TC(u)\over\sqrt{2\expected}},\quad 
Y^{n_T}_k={N(J_k,u)-q_TC(u)\over\sqrt{2\expected}}
$$
Our strategy to establish \eqref{eq:main}, is to prove
\begin{align}
\sum_{k=1}^{n_T}X^{n_T}_k&\Rightarrow \mathcal{N}(0,1)\label{eq:x},\\
\sum_{k=1}^{n_T}Y^{n_T}_k&\overset{P}\to 0,\label{eq:y}
\end{align}
where $\overset{P}\to$ stands for convergence in probability.

%%%%%%%%%%%%%%%%%%%%%%
In both cases we need a bound for the (factorial) third moment of $N(I,u)$. 
The next result gives such a bound.
As the proof requires the use of the Rice formula for the third (factorial) moment, it is deferred to Subsection \ref{53}.
\begin{lemma}\label{lemma:third}
For $R>0$, there exist constants $A,B$, independent of $R$ and $u$, such that
$$
\E \big[N(R,u)^{[3]}\big]\leq R^3(A+Bu^3)e^{-u^2/2}.
$$
\end{lemma} 

% Second moment
\subsection{Proof of \eqref{eq:y} - Small blocks}
The convergence in \eqref{eq:y} is a consequence of the following lemma.
\begin{lemma}\label{lemma:short}
{\rm (a)} Let $\ell(T)=o(T)$ as $T\to\infty$. We have
$$
{N(\ell(T),u)-\ell(T)C(u)\over \sqrt{\expected}}\overset{P}\to 0,\quad (T\to\infty).
$$
\par\noindent
{\rm (b)} Under the mixing condition \eqref{eq:mixing}, we have 
$$
\E\left(\sum_{k=1}^nY^n_k\right)^2\to 0,\quad (T\to\infty).
$$
\end{lemma}
\begin{proof}In order to verify (a), we write
$$
\E\left({N(\ell(T),u)-\ell(T)C(u)\over \sqrt{\expected}}\right)^2={\var N(\ell(T),u)\over \expected}\sim {\ell(T)\over T}\to 0.
$$
To see (b), we use a moment inequality under mixing (see 1.2.2 in \cite{mixing}), that in our context reads
$$
|\E(Y^n_k,Y^n_{k+\ell})|\leq 8 \alpha(\ell p_T)^{1/3}\left(\E(|Y^n_1|^3)\right)^{2/3}.
$$
So, applying Lemma \ref{lemma:third} and the mixing rate, we obtain
\begin{multline*}
|\E(Y^n_k,Y^n_{k+\ell})|\leq 8 \ell^{-\mu/3}T^{-b\mu/3}\left[{(A+Bu^3)q_T^{3}e^{-u^2/2}\over \expected^{3/2}}\right]^{2/3}
\leq 8 \ell^{-\mu/3}T^{-b\mu/3}\left[{(A+Bu^2)q_T^{2}e^{-u^2/3}\over \expected}\right]\\
\leq8 \ell^{-\mu/3}\left[(A+Bu^2)T^{2c{-b\mu/3}- \gamma/3-2/3}\right]
\leq 8 \ell^{-\mu/3}T^{2b{-b\mu/3}- \gamma/3-2/3+\epsilon}.
\end{multline*}
Note that as $\mu>4$ the series $\sum_{\ell=1}^{\infty}\ell^{-\mu/3}<\infty$. So
\begin{multline*}
\E\left(\sum_{k=1}^nY^n_k\right)^2=n\E(Y^n_1)^2+2(n-1)\E(Y^n_1Y^n_2)+\cdots+4\E(Y^n_1Y^n_{n-2})+2\E(Y^n_1Y^n_{n-1})\\
\leq \frac1{\expected}K n q_Te^{-u^2/2}+
16nT^{2b{-b\mu/3}-\gamma/3-2/3+\epsilon}
\sum_{\ell=1}^{\infty}\ell^{-\mu/3}
\\
\leq \frac1{\expected}K n q_Te^{-u^2/2}+
KT^{1/3+b-b\mu/3-\gamma/3+\epsilon}
\to 0,
\end{multline*}
because $\epsilon>0$ can be chosen arbitrarily small and 
$\mu b>1$. 
This concludes the proof of the Lemma.
\end{proof}

\subsection{Proof of \eqref{eq:x} - Big blocks}
In order to prove \eqref{eq:x}, we use the Central Limit Theorem
proved by Lindeberg's method as presented in Theorem 1 in \cite{bdlr}. %\marginpar{\textcolor{red}{5.Check: escribirlo o dar algun detalle}}
The next three lemmas verify the hypothesis of Theorem 1 in \cite{bdlr}.
\begin{lemma}The mixing rate of the process $\X$ implies that as $T\to\infty$:
$$
\sum_{k=1}^{n_{{ T}}}|\cov(e^{it(X^n_1+\cdots+X^n_{k-1})},e^{itX^n_k})| \to 0.
$$
\end{lemma}
\begin{proof}
We use  Kolmogorov-Rozanov inequality  as in the proof of (c) in Lemma \ref{lemma:covariance}.
Since $1-b-\mu c\leq 1-b-\mu b< -b<0$ by \eqref{eq:rates}, we have
$$
\sum_{k=1}^{n_{{ T}}}|\cov(e^{it(X^n_1+\cdots+X^n_{k-1})},e^{itX^n_k})|\leq \mathbf C n_{{T}}\alpha(q_T)\sim
T^{1-b-\mu c}\to 0.
$$
\end{proof}
%% LA VARIANZA ES EQUIVALENTE A 2 LA LONGITUD, ARRIBA DA 2C(u)p_T y abajo 2 T C(u) 
\begin{lemma}
As $T\to\infty$:
$$
\sum_{k=1}^n\var X^n_k=n\var X^n_1\sim {np_TC(u)\over \expected}\to 1.
$$
\end{lemma}
In the next lemma we check hypothesis $H_\delta$ of Theorem 1 in \cite {bdlr}
with $\delta=1$.
\begin{lemma}[$H_\delta$ for $\delta=1$]\label{lemma:delta}
As $n\to\infty$, we have
$$
A_n=\sum_{k=1}^n\E|X^n_k|^{3}\leq {n\over \expected^{3/2}}\E|N(p_T,u)-p_TC(u)|^{3}\to 0.
$$
\end{lemma}
\begin{proof}
Denote $N=N(I_1,u)$.
We use the following rough bound 
$$
\E|N-\E N|^3\leq 8\left[\E(N^3)+(\E N)^3\right]\leq 16 \E(N^3).
$$
where we rely on $(A+B)^3\leq 8(A^3+B^3)$, for $A,B>0$; and Jensen's inequality for the convex function
$f(x)=|x|^3$.
Then
$$
\E|X^n_1|^{3}\leq \frac{16}{\expected^{3/2}}\E(N^3).
$$
The proof is then split into several steps.
\begin{enumerate}
\item
We expand the third factorial moment of the crossings
$$
\E N^{[3]}:=\E(N(N-1)(N-2))=\E N^3-3\E N^2+2\E N.
$$
\item We verify
$$
{n\E N\over \expected^{3/2}}={np_TC(u)\over \expected^{3/2}}\sim{1\over\sqrt{\expected}}
\to 0.
$$
\item We verify
\begin{multline*}
{n\E N^2\over \expected^{3/2}}={n(\var N+(\E N)^2)\over \expected^{3/2}}
\sim {n(\E N+(\E N)^2)\over \expected^{3/2}}
\sim {n(\E N)^2\over \expected^{3/2}}\\
\sim T^{1-b}p_T^2C(u)^2T^{-3\gamma/2}\sim T^{1-b}T^{2b}T^{2(\gamma-1)}T^{-3\gamma/2}
\sim T^{-1+b+\gamma/2}
\to 0,
\end{multline*}
as
$b<\gamma/4$.
\item It remains to prove 
\begin{equation*}
{n\E N^{[3]}\over \expected^{3/2}}\to 0.
\end{equation*}
This step is a consequence of Lemma \ref{lemma:third}.  
In fact,
\begin{equation*}
{n\E(N^{[3]})\over \expected^{3/2}}
\leq {np_T^3(A+Bu^3)e^{-u^2/2}\over T^{3\gamma/2}}
\leq {KT^{1+2b+\gamma-1-3\gamma/2+\epsilon}}
\leq T^{2b-\gamma/2+\epsilon}\to 0,
\end{equation*}
because $4b<\gamma$ and  $\epsilon$ can be taken arbitrarily small.
\end{enumerate}
\end{proof}

%%%%%%%%%%%%%%%%%%%%%%%%%%%%%%
% \appendix
%%%%%%%%%%%%%%%%%%%%%%%%%%%%%%
% \section{Appendix: Additional proofs}\label{app1}
\subsection{Proof of Lemma \ref{lemma:third}}\label{53}
Below, based on Rice formula, we present the proof for the bound on the third moment given in Lemma \ref{lemma:third}. 
In this section $K,K'$ stands for meaningless constants whose value may change from line to line.
\begin{proof}[Proof of Lemma \ref{lemma:third}] 
Let $\bu=(u,u,u)$, $\C=\{X(s_1)=X(s_2)=X(s_3)=u\}$ and 
$$
A(s_1,s_2,s_3)=\E(|\dot{X}(s_1)\dot{X}(s_2)\dot{X}(s_3)|\mid \C)p_{X(s_1),X(s_2),X(s_3)}(\bu).
$$
Rice formula \eqref{eq:rice} for $n=3$ states
\begin{equation*}
\E (N(R,u)^{[3]})=\iiint_{[0,R]^3}A(s_1,s_2,s_3)ds_1ds_2ds_3.
\end{equation*}
As the process is stationary, we have $A(s_1,s_2,s_3)=A(0,s_2-s_1,s_3-s_1)$, so  we change variables according to 
$$
s=s_1,\quad h=s_2-s_1,\quad k=s_3-s_2.
$$
We have
$$
\aligned
\E (N^{[3]}(R,u))
&=6\iiint_{\{0\leq s_1\leq s_2\leq s_3\leq R\}}A(0,s_2-s_1,s_3-s_1)ds_1ds_2ds_3\\
&=6\iint_{\{0\leq h+k\leq R\}}\left(\int_0^{R-h-k}A(0,h,h+k)ds\right)dhdk
\\
&=6\int_0^{R}dk\int_{0}^{R-k}(R-h-k)A(0,h,h+k)dh
\\
&\leq 
6\int_0^{\delta}dk\int_{0}^{\delta}(R-h-k)A(0,h,h+k)dh\\
&\quad+ 
12\int_{\delta}^{R}dk\int_{0}^{\delta}(R-h-k)A(0,h,h+k)dh\\
&\quad+ 
6\int_{\delta}^{R}dk\int_{\delta}^{R-k}(R-h-k)A(0,h,h+k)dh=:I_1+I_2+I_3.\\
\endaligned
$$
In conclusion, we have to bound the three integrals that appear above. 
The proof is divided into several steps. \\

\noindent{\it First step:} General facts. 
Since $(X(0),X(h),X(h+k))$ is non-degenerated by hypothesis, 
its covariance matrix 
$$
\Sigma=\left[
\begin{matrix}
1&r(h)&r(h+k)\\
r(h)&1&r(k)\\
r(h+k)&r(k)&1\\
\end{matrix}
\right].
$$
is not singular.

The density is
$$
p_{X(0),X(h),X(h+k)}(\bu)=\frac1{(2\pi)^{3/2}\sqrt{\Delta}}\exp\left({-\frac12\bu\Sigma^{-1}\bu^t}\right).
$$

\noindent{\bf Claim:}
We have
\begin{equation}\label{lemma:density}
\bu\Sigma^{-1}\bu^t> u^2. 
\end{equation}
\newcommand{\vp}{\kappa}
\begin{proof}[Proof of the Claim] As $\Sigma$ is symmetric and positive definite,
we denote its three eigenvalues by
$$
0<\vp_1\leq\vp_2\leq\vp_3,
$$ 
and introduce the diagonal matrix $D=\text{diag}(\vp_1,\vp_2,\vp_3)$. Observe that
$$
\vp_1+\vp_2+\vp_3=\text{trace}(\Sigma)=3,
$$
so $\vp_3<3$.
If $P$ denotes the corresponding orthogonal matrix, such that
$\Sigma P=PD$,
then
$\Sigma^{-1}=PD^{-1}P^t$.
So, denoting $\bv=(v_1,v_2,v_3)=\bu P$, 
\begin{equation*}
\bu\Sigma^{-1}\bu^t
=\bv D^{-1}\bv^t
={v_1^2\over\vp_1}+{v_2^2\over\vp_2}+{v_3^2\over\vp_3}\geq 
\frac1{\vp_3}\bv\bv^t>\frac13\bu PP^t\bu^t=\frac13 3u^2=u^2,
\end{equation*}
completing the proof.
\end{proof}
Hence, the exponential factor in the density is bounded by $e^{-u^2}$. 
The denominator of this density must be bounded jointly with the remaining factors in the integrals.

%%%%%%%%%%%%%%%%%%%%%%
%On the regression in $\R^3$
%%%%%%%%%%%%%%%%%%%%%%

We now specialize Lemma \ref{lemma:1} to the case $n=3$.
We use the notation
$$
\matrixalpha
=
\Sigma_{10}\Sigma^{-1}
=
\left[
\begin{matrix}
\Psi_1\\ \Psi_2\\ \Psi_3
\end{matrix}
\right]=
\left[
\begin{matrix}
\alpha_{11}&\alpha_{{12}}&\alpha_{{13}}\\ 
\alpha_{{21}}&\alpha_{22}&\alpha_{{23}}\\ 
\alpha_{{31}}&\alpha_{{32}}&\alpha_{33}
\end{matrix}
\right],
$$
$$
r_{ij}=\E X_iX_j,\quad 
\dr_{ij}=\E(\dx_iX_j),\quad
\ddr_{ij}=\E(\dx_i\dx_j).
$$
$$
\Sigma_{10}=\left[
\begin{matrix}
\dr_{11}&\dr_{12}&\dr_{13}\\ 
\dr_{21}&\dr_{22}&\dr_{23}\\ 
\dr_{31}&\dr_{32}&\dr_{33}\\
\end{matrix}
\right]
=\left[
\begin{matrix}
0&\dr_{12}&\dr_{13}\\ 
{-}\dr_{12}&0&\dr_{23}\\ 
{-}\dr_{13}&{-}\dr_{32}&0\\
\end{matrix}
\right]
=\left[
\begin{matrix}
0&{-}\dr(h)&{-}\dr(h+k)\\ 
\dr(h)&0&{-}\dr(k)\\ 
\dr(h+k)&\dr(k)&0\\
\end{matrix}
\right],
$$
$$
\Sigma_{11}=\left[
\begin{matrix}
\ddr_{11}&\ddr_{12}&\ddr_{13}\\ 
\ddr_{21}&\ddr_{22}&\ddr_{23}\\ 
\ddr_{31}&\ddr_{32}&\ddr_{33}\\
\end{matrix}
\right]=\left[
\begin{matrix}
-1&-\ddr(h)&-\ddr(h+k)\\ 
-\ddr(h)&-1&-\ddr(k)\\ 
-\ddr(h+k)&-\ddr(k)&-1\\
\end{matrix}
\right].
$$
Recall that from Lemma \ref{lemma:1} we have
\begin{equation*}
\E(|\dot{X}(s_1)\dot{X}(s_2)\dot{X}(s_3)|\mid \C)
=\E(|Y_1-\Psi_1\bu||Y_2-\Psi_2\bu||Y_3-\Psi_3\bu|).
\end{equation*}

\noindent{\it Second step:} First integral. 
We make the Taylor expansion of each component of the integrand as $(h,k)\to(0,0)$ 
in order to prove that the integral is convergent. 
As a consequence we obtain a bound of the desired form for this term. 
%{On the determinant close to the origin}
%{the density}
The following result is taken from Proposition 5.9, equation (5.22) in \cite{aw}. 
We have
$$
\Delta:=\det(\Sigma)=1+2r(h)r(k)r(h+k)-r(h)^2-r(k)^2-r(h+k)^2,
$$
that, for $(h,k)\to(0,0)$ satisfies
\begin{equation}\label{eq:delta}
\Delta%=\det(\Sigma)
\sim\left({24d-1\over{4}}\right){{h}}^2{{k}}^2(h+k)^2.
\end{equation}

%{On the regression coefficients close to the origin}
%{On the conditional expectation}

\noindent{\bf Claim:} (On the regression coefficients in $\matrixalpha$.) 
{Denote $\alpha_i=\Psi_i (1,1,1)^t=\alpha_{1i}+\alpha_{2i}+\alpha_{3i},\ (i=1,2,3)$. }
As $(h,k)\to 0$ we have
\begin{align*}
\alpha_1&\sim{-}\left( {8d^2+10e\over 24d-1}\right)h(h+k)(2h+k),\\
\alpha_2&\sim \left({8d^2+10e\over24d-1}\right)hk(h-k),\\
\alpha_3&\sim {-}\left({8d^2+10e\over24d-1}\right)k(h+k)(2k+h).
\end{align*}

\begin{proof}[Proof of the Claim.]
We have
$$
\Sigma^{-1}=\frac1\Delta
\left[
\begin{matrix}
1-r^2(k)		&r(h+k)r(k)-r(h)		&r(h)r(k)-r(h+k)\\
r(h+k)r(k)-r(h)		&1-r^2(h+k)			&r(h)r(h+k)-r(k)\\
r(h)r(k)-r(h+k)		&r(h)r(h+k)-r(k)		&1-r^2(h)
\end{matrix}
\right],
$$
so 
\begin{align}\label{eq:DA}
\Delta\alpha_1	&=-[1-r(h+k)][1+r(h+k)-r(h)-r(k)]\dr(h) 
		-(1-r(h))(1+r(h)-r(k)-r(h+k))\dr(h+k),\notag\\
\Delta\alpha_2	&=(1-r(k))(1+r(k)-r(h)-r(h+k))\dr(h)
		-(1-r(h))(1+r(h)-r(k)-r(h+k))\dr(k),\notag\\
\Delta\alpha_3	&=(1-r(k))(1+r(k)-r(h)-r(h+k))\dr(h+k)
		+(1-r(h+k))(1+r(h+k)-r(h)-r(k))\dr(k).
\end{align}
We perform the Taylor expansions for  $\alpha_i=\alpha_i({{h}},{{k}})$ ($i=1,2,3$): 
\begin{align*}
\Delta\alpha_1({{h}},{{k}})	&\sim{-}\left(\frac{4 d^{2}+5 e}2\right)h^3k^2(h+k)^3(2h+k),\\
\Delta\alpha_2({{h}},{{k}})	&\sim2d^2h^3k^3(h+k)^2(h-k),\\
\Delta\alpha_3({{h}},{{k}})	&\sim{-}\left(\frac{4d^{2}+5e}{2}\right)h^3k^3(h+k)^2(k-h),
\end{align*}
that in view of \eqref{eq:delta}, give the results.
\end{proof}

%On the variance of $Y$ close to the origin
\noindent{\bf Claim:} For $i=1$ and $i=3$ we have %(on the variance of $Y_i$)
$$
\var Y_i\sim {4(16d^2 - d-10e)\over 24d-1} (h^2+hk+k^2),
$$
and for $i=2$,
\begin{equation*}
 \var Y_2\sim {4d(16d - 1)\over 24d-1} (h^2+hk+k^2).
 \end{equation*}
 
\begin{proof}[Proof of the Claim.]
We depart from
\begin{equation*}
1=\var \dx(s_1)=\var Y_1-(\alpha_{12}\dr_{12}+\alpha_{13}\dr_{13}),
\end{equation*}
and expand the expression
\begin{equation*}
\Delta(\alpha_{12}\dr_{12}+\alpha_{13}\dr_{13})(h,k)
=-\left(16d^{2} - d - 10e\right)h^2k^2(h+k)^2(h^2+hk+k^2)
+ \frac{1}{4}{\left(24d - 1\right)} h^2 k^{2}(h+k)^2.
\end{equation*}
We conclude
\begin{equation*}
\var Y_1=1-(\alpha_{12}\dr_{12}+\alpha_{13}\dr_{13})(h,k)\sim
{4(16d^2 - d-10e)\over 24d-1} (h^2+hk+k^2).
\end{equation*}
Similar computations hold for $i=3$. For $i=2$, we compute
 \begin{equation*}
 1=\var \dx_i=\var Y_2-(\alpha_{21}\dr_{21}+\alpha_{23}\dr_{23}).
 \end{equation*}
Besides,
 \begin{equation*}
 \Delta(\alpha_{21}\dr_{21}+\alpha_{23}\dr_{23})(h,k)
 \sim \frac14(24d-1)h^2k^2(h+k)^2
 -(16d^{2} - d) h^2 k^2(h+k)^2(h^2+hk+k^2),
 \end{equation*}
and in consequence
 \begin{equation*}
 \var Y_2=1-(\alpha_{21}\dr_{21}+\alpha_{23}\dr_{23})(h,k)\sim
 {4d(16d - 1)\over 24d-1} (h^2+hk+k^2).
 \end{equation*}
This proves the claim.
\end{proof}
\noindent{\bf Claim:}
For $u>1$ there exist constants $A_1,B_1$ such that
$$
\E(|\dot{X}(s_1)\dot{X}(s_2)\dot{X}(s_3)|\mid \C)\leq (A_1+B_1u^3) E(h,k),
$$
with $E(h,k)=o(h^2k^2(h+k)^2)$ as $(h,k)\to(0,0)$.

\begin{proof}[Proof of the Claim.]
We have
\begin{equation}\label{eq:tr}
\E(|\dot{X}(s_1)\dot{X}(s_2)\dot{X}(s_3)|\mid \C)=
\E(|(Y_1+\alpha_1u)(Y_2+\alpha_2u)(Y_3+\alpha_3u)|)
\leq J_1+J_2u+J_3u^2+|\alpha_1\alpha_2\alpha_3|u^3,
\end{equation}
being
\begin{align*}
J_1&=\E(|Y_1Y_2Y_3|)
=\iiint_{\R^3}|y_1y_2y_3|
\frac1{(2\pi)^{3/2}\sqrt{\det(\Sigma_Y)}}e^{-\frac12y\Sigma_Y^{-1}y^t}dy,\\
J_2&=\sum_{i=1}^3|\alpha_i|\E\left(\prod_{k\neq i}^3|Y_k|\right)=\sum_{i=1}^3|\alpha_i|\iiint_{\R^3}\left(\prod_{k\neq i}^3|y_k|\right)
\frac1{(2\pi)^{3/2}\sqrt{\det(\Sigma_Y)}}e^{-\frac12y\Sigma_Y^{-1}y^t}dy,\\
J_3&=\sum_{i=1}^3\left(\prod_{k\neq i}^3|\alpha_k|\right)\E(|Y_i|)=\sum_{i=1}^3\left(\prod_{k\neq i}^3|\alpha_k|\right)\iiint_{\R^3}|y_i|
\frac1{(2\pi)^{3/2}\sqrt{\det(\Sigma_Y)}}e^{-\frac12y\Sigma_Y^{-1}y^t}dy.
\end{align*}
To deal with them we perform the change of variables
$$
y_j=\gamma_jz_j,\qquad \gamma_j=\prod_{i\neq j}^3(s_i-s_j)^2;\qquad j=1,2,3.
$$
Furthermore, we use the equivalents above and 
the fact that (see pg. 147 in \cite{aw})
$$
\det(\Sigma_Y)\sim K\prod_{1\leq i<j\leq 3}(s_j-s_i)^6.
$$

The first integral $J_1$ is the same as the one that
appears for the case $u=0$ in Prop. 5.10 in \cite{aw}: 
\begin{equation*}
J_1=\prod_{1\leq i<j\leq 3}(s_j-s_i)^8\iiint_{\R^3}|z_1z_2z_3|
\frac1{(2\pi)^{3/2}\sqrt{\det(\Sigma_Y)}}
e^{-\frac12 zGz'}dz
\sim Kh^5k^5(h+k)^5,
\end{equation*}
where $G=D\Sigma^{-1}_YD$ and $D=diag(\gamma_1,\gamma_2,\gamma_3)$. 
The second term $J_2$ is
\begin{multline*}
J_2=|\alpha_1|\iiint_{\R^3}|z_2z_3|
\frac1{(2\pi)^{3/2}\sqrt{\det(\Sigma_Y)}}e^{-\frac1{2\det(\Sigma_Y)}zGz'}dz
+|\alpha_2|\iiint_{\R^3}|z_1z_3|
\frac1{(2\pi)^{3/2}\sqrt{\det(\Sigma_Y)}}e^{-\frac1{2\det(\Sigma_Y)}zGz'}dz\\
+|\alpha_3|\iiint_{\R^3}|z_1z_2|
\frac1{(2\pi)^{3/2}\sqrt{\det(\Sigma_Y)}}e^{-\frac1{2\det(\Sigma_Y)}zGz'}dz\\
\sim K_1  (s_3-s_2)^5(s_2-s_1)^4(s_3-s_1)^4(2s_3-s_2-s_1)
\quad+K_2  (s_3-s_1)^5(s_3-s_2)^4(s_1-s_2)^4(2s_2-s_3-s_1)\\
\quad+K_3  (s_2-s_1)^5(s_1-s_3)^4(s_2-s_3)^4(2s_3-s_2-s_1)\\
=h^4k^4(h+k)^4[K_1k(2h+k)+K_2(h+k)(h-k)+K_3h(2h+k)].
\end{multline*}
The third term $J_3$ is
\begin{multline*}
J_3=|\alpha_2\alpha_3|\iiint_{\R^3}|z_1|
\frac1{(2\pi)^{3/2}\sqrt{\det(\Sigma_Y)}}e^{-\frac1{2\det(\Sigma_Y)}zGz'}dz
+|\alpha_1\alpha_3|\iiint_{\R^3}|z_2|
\frac1{(2\pi)^{3/2}\sqrt{\det(\Sigma_Y)}}e^{-\frac1{2\det(\Sigma_Y)}zGz'}dz\\
+|\alpha_1\alpha_2|\iiint_{\R^3}|z_3|
\frac1{(2\pi)^{3/2}\sqrt{\det(\Sigma_Y)}}e^{-\frac1{2\det(\Sigma_Y)}zGz'}dz\\
\sim h^3k^3(h+k)^3[K_1h(h+k)(h-k)(2k+h)
+K_2hk(2k+h)(2h+k)+K_3k(h+k)(2h+k)].
\end{multline*}
Finally,
$$
|\alpha_1\alpha_2\alpha_3|\sim K h^2k^2(h+k)^2(2k+h)(2h+k)(h-k).
$$
Gathering all these terms  
and definig $A_1=J_1$, $B_1=J_2+J_3+|\alpha_1\alpha_2\alpha_3|$ 
we conclude the proof of the Claim.
\end{proof}

%Bound of the first integral
In conclusion, 
we have
\begin{equation*}
A(0,h,h+k)\leq (A_1+B_1u^3)\frac{E(h,k)}{\sqrt{\Delta}}e^{-\frac12u^2},\quad(h,k)\to(0,0),
\end{equation*}
and $\frac{E(h,k)}{\sqrt{\Delta}}=o(hk(h+k))$ as $(h,k)\to(0,0)$. 
Consequently
$$
I_1\leq 
\int_0^{\delta}dk\int_{0}^{\delta}
(R-h-k)
(A_1+B_1u^3)\frac{E(h,k)}{\sqrt{\Delta}}e^{-\frac12u^2}
dhdk
\leq K\,R(A_1+B_1u^3)e^{-\frac12u^2},
$$

%%%%%%%%%%%%%%%%%%
\noindent{\it Third step:} Second integral. 
We have $h<\delta<k\leq R$. Then, we expect $Y_1$ and $Y_2$ to be small,
as $s_1$ is close to $s_2$.

We proceed as in \eqref{eq:tr}. 
Besides, we have $\var Y_i\leq 1$ and 
when $h$ is bounded away from $0$ is easy to bound $\alpha_i$, $i=1,2,3$ 
by a constant, see \eqref{eq:DA}. 
Thus, also $J_1$, $J_2$ and $J_3$ are bounded by constants. 

The result follows provided the 
convergence of the integral at $h=0$. 

If we consider
$
\Delta=\Delta(h,k),
$
we have $$
\partial_h\Delta(0,k)=\partial_k\Delta(0,k)=\partial_{hk}\Delta(0,k)=\partial_{kk}\Delta(0,k)=0,
$$
and
$$
\Delta(h,k)\sim \left[1-r(k)^2-\dr(k)^2\right]h^2,\quad(h,k)\to(0,k).
$$

%On the expansion of $\alpha_i$ when $(h,k)\to (0,k)$.
Now, the Taylor expansion gives
$$
\Delta\alpha_1\sim 
[(r(k)-1)\ddot{r}(k)-\dot{r}(k)^2-r(k)+1]\frac{h^3}2,
$$
that in view of the expansion of the determinant gives
$$
\alpha_1\sim C_1(k) h,
$$
where $C_1(k)$ is a continuous and bounded function of $k\geq\delta$.
{By similar computations} we have 
$$
\alpha_2\sim C_2(k) h,
$$
being $C_2(k)$ continuous and bounded on $k\geq\delta$. 
%{On the expansion of $\sigma_i$ when $(h,k)\to (0,k)$.}
For the variance of $Y_i$, %in this case, 
we have 
$
\var Y_1=1-(\alpha_{12}\dr_{12}+\alpha_{13}\dr_{13}).
$ 
We have
$$
\Delta
(\alpha_{12}\dr_{12}+\alpha_{13}\dr_{13})
={\left(1-r(k)^{2} - \dr(k)^{2} \right)} h^{2}-{\left(r(k) \dr(k) + \dr(k) \ddr(k)\right)} h^{3},
$$
so
$$
\alpha_{12}\dr_{12}+\alpha_{13}\dr_{13}
=1-
{
r(k) \dr(k) + \dr(k) \ddr(k)
\over
1-r(k)^2-\dr(k)^2
}
 h.
$$
So
$$
\var Y_1\sim C'_1(k)h.
$$
By similar arguments, we have 
$$
\var Y_2\sim C'_2(k) h.
$$
Both $C'_1$ and $C'_2$ are continuous and bounded on $k\geq\delta$. 
% Turning back to $I_2$. 
We have the bound
\begin{equation*}
\E(|\dot{X}(s_1)\dot{X}(s_2)\dot{X}(s_3)|\mid \C)\leq \prod_{k=1}^3\left(\E(|\dot{X}(s_i)|^3\mid\C)\right)^{1/3}.
\end{equation*}
The first two expectations depend on $h$, so, for $i=1,2$
$$
\E(|\dot{X}(s_i)|^3\mid\C)=\E|Y_i-\alpha_iu|^3\leq 8\var Y_i^{3/2}+8|\alpha_i|^3u^3
\sim K h^{3/2}+K'u^3h^3.
$$
Furthermore,
$$
\E(|\dot{X}(s_3)|^3\mid\C)=\E|Y_3-\alpha_3u|^3\leq K+K'u^3.
$$
As $\sqrt{\Delta}\sim \sqrt{1-r(k)^2-\dr(k)^2}h$ when $(h,k)\to (0,k)$, we obtain that for these values 
$$
A(0,h,h+k)\leq (K+K' u^3)e^{-u^2/2}.
$$
In conclusion 
$$
I_2\leq (A_2+B_2u^3)R^2e^{-u^2/2}.
$$\\

\noindent{\it Fourth step:} Third integral.  
We consider  the off-diagonal term.\\

\noindent{\bf Claim:}
When $\delta<\min(h,k)$, we have
\begin{equation}\label{eq:big}
\E(|\dot{X}(s_1)\dot{X}(s_2)\dot{X}(s_3)|\mid \C)\leq A_3+B_3u^3,
\end{equation}
where $A_3$, $B_3$ are constants, depending on $\delta$.

\begin{proof}[Proof of the Claim.]
\begin{equation*}
\E(|\dot{X}(s_1)\dot{X}(s_2)\dot{X}(s_3)|\mid \C)\leq \prod_{k=1}^3\left(\E(|\dot{X}(s_i)|^3\mid\C)\right)^{1/3}.
\end{equation*}
We have
$$
\E(|\dot{X}(s_i)|^3\mid\C)=\E|Y_i-\alpha_iu|^3\leq 8\E|Y_i|^3+8|\alpha_i|^3u^3.
$$
By (a) in Lemma \ref{lemma:1}, we know that $\var Y_i\leq 1$. 
We need some rough bounds.
Based on \eqref{eq:DA} 
and (a) in Lemma \ref{lemma:covariance},
we obtain that
$
|\Delta\alpha_i|\leq 16.
$
As in the domain of integration, there exists $\Delta_0>0$ such that $\Delta\geq\Delta_0$, 
we obtain that
$$
|\alpha_i|\leq{16\over\Delta_0},\quad i=1,2,3.
$$
This concludes the proof of the claim.
\end{proof}
Taking into account \eqref{eq:big} and the bound of the density of \eqref{lemma:density}:
$$
I_3\leq K
\int_{\delta}^{R}dk\int_{\delta}^{R-k}(R-h-k)(A+Bu^3)e^{-\frac12u^2}dh
\leq 
K R^3(A+Bu^3)e^{-\frac12u^2}.
$$

This concludes the proof of Lemma \ref{lemma:third}.
\end{proof}

\section{Proof of Theorem \ref{teo:vr}}\label{theorem3}
\begin{proof}
We begin the proof considering the up-crossings.
Then, relation \eqref{eq:key} reads as
$$
{\var U(T,u)\over \E\, U(T,u)}\to 1,
$$
uniformly in $T\in[t_0,\infty)$. 
The normalized number of up-crossings is 
$$
Z(T,u)={U(T,u)-\E U(T,u)\over\sqrt{\E U(T,u)}},
$$
and denote its probability distribution by
$$
F(T,u)=F_{Z(T,u)}(\cdot). 
$$
Denote by $\pi(F,G)$ the Prohorov distance between the probability measures induced by the distributions
$F$ and $G$ in $\R$.
Denote now by $N(\ell)$ a Poisson random variable with parameter $\ell$,
and 
$$
Y(\ell)={N(\ell)-\ell\over\sqrt{\ell}},
$$
and the corresponding distribution
$$
G(\ell)=G_{Y(\ell)}(\cdot).
$$
By the Central Limit Theorem for Poisson random variables, we know that
$$
\pi(G(\ell),\Phi)\to 0,\quad (\ell\to\infty), 
$$
where $\Phi$ stands for the standard normal distribution function. 
This means that there exists $\ell_n$ such that
$$
\pi(G(\ell),\Phi)
<1/(2n),\quad\textrm{ for }\ell\geq \ell_n.
$$
We have obtained a sequence $\ell_n\to\infty$.
We now observe that we are under the hypothesis of the Volkonski\u{\i}-Rozanov Theorem 
that states that, as processes indexed by $\ell>0$, we have
$$
U(\ell C(u)^{-1},u)\Rightarrow P(\ell),\quad (u\to\infty),
$$
where $P=\{P(\ell)\colon \ell\geq 0\}$ denotes a standard Poisson process. %(see Th. 10.1 in \cite{aw}).
We then have, for each $n$, 
$$
U(\ell_nC(u)^{-1},u)\Rightarrow P(\ell_n),\quad (u\to\infty),
$$
that is equivalent to 
$$
{U(\ell_nC(u)^{-1},u)-\ell_n\over\sqrt{\ell_n}}\Rightarrow{P(\ell_n)-\ell_n\over\sqrt{\ell_n}}=Y
(\ell_n),\quad (u\to\infty),
$$
In consequence, there exists $u_n\to\infty$ such that
$$
\pi(F(\ell_nC(u)^{-1},u),G(\ell_n))<\frac1{2n},\text{ for }u\geq u_n.
$$
Denote $T_n=\ell_nC(u_n)^{-1}$. 
As $\ell_n$ and $u_n$ are increasing it follows that $T_n\to\infty$. 
We obtain that
$$
\pi(F(T_n,u_n),\Phi)\leq\pi(F(T_n,u_n),G(\ell_n))+\pi(G(\ell_n),\Phi)<\frac1n,
$$
concluding the proof for up-crossings.
Now, based on Remark \ref{remark:up}, we have
$$
\frac{N(T_n,u_n)-\lambda(T_n,u_n)}{\sqrt{2\lambda(T_n,u_n)}}
={U(T_n,u_n)-\E U(T_n,u_n)\over\sqrt{\E U(T_n,u_n)}}+
{r(T_n,u_n)\over\sqrt{\E U(T_n,u_n)}},
$$
where $|r(T_n,u_n)|\leq 1/2$. As $\E U(T_n,u_n)\to\infty$, the result follows.
\end{proof}

% 

%%%%%%%%%%%%%%%%%%%%%%%%%%%%%%%%%%%%%%%%%%%%%%%



%% References


%\bibliographystyle{elsarticle-num}
%\bibliography{<ref2>}






% \end{thebibliography}

\end{document}